\documentclass[11pt]{article}
\usepackage{amsfonts,amsmath,latexsym,color,epsfig}
\usepackage{here}
\newtheorem{theorem}{Theorem}[section]
\newtheorem{lemma}[theorem]{Lemma}

\newtheorem{corollary}[theorem]{Corollary}

\newtheorem{claim}[theorem]{Claim}

\newtheorem{problem}[theorem]{Problem}

\newenvironment{proof}
      {\medskip\noindent{\bf Proof:}\hspace{1mm}}
      {\hfill$\Box$\medskip}

\input{epsf}

\def\qed{\ifvmode\mbox{ }\else\unskip\fi\hskip 1em plus 10fill$\Box$}
\def\eps{\varepsilon}

\long\def\ignore#1{}

\def\rc{\advance\leftskip by 0pt plus 40em\rightskip=\leftskip
  \parfillskip=0pt \spaceskip=.3333em \xspaceskip=.5em
  \pretolerance=9999 \tolerance=9999 \hyphenpenalty=9999
  \exhyphenpenalty=9999}

\begin{document}

\title{Decomposition of a cube into\\ nearly equal smaller cubes} \author{Peter Frankl
  \thanks{R\'enyi Institute, Hungarian Academy of Sciences, H--1364 Budapest, POB 127, Hungary. Email: {\tt peter.frankl@gmail.com}}
  \and Amram Meir\thanks{York University, 1121 Steeles Ave West, apt. PH-206, Toronto, Ontario, M2R-3W7, Canada. Email: {\tt ameir@rogers.com}}
  \and J\'anos Pach\thanks{R\'enyi Institute and EPFL, Station 8, CH--1014 Lausanne, Switzerland.
    Email: {\tt pach@cims.nyu.edu}. Supported by Swiss National Science Foundation Grants 200020-
144531 and 200021-137574.}}

\date{}

\maketitle

\begin{abstract}
Let $d$ be a fixed positive integer and let $\eps>0$. It is shown that for every sufficiently large $n\ge n_0(d,\eps)$, the $d$-dimensional unit cube can be decomposed into exactly $n$ smaller cubes such that the ratio of the side length of the largest cube to the side length of the smallest one is at most $1+\eps$. Moreover, for every $n\ge n_0$, there is a decomposition with the required properties, using cubes of at most $d+2$ different side lengths. If we drop the condition that the side lengths of the cubes must be roughly equal, it is sufficient to use cubes of two different sizes.
\end{abstract}

\section{Introduction}

It was shown by Dehn~\cite{De03} that, for $d\ge 2$, in any decomposition (tiling) of the $d$-dimensional unit cube into finitely many smaller cubes, the side length of every participating cube must be rational. Sprague~\cite{S40} proved that in the plane there are infinitely many decompositions consisting of pairwise incongruent squares. Such a decomposition is called {\em perfect}. Brooks, Smith, Stone, and Tutte~\cite{BSST40} developed a method to generate all perfect decompositions of squares, by reformulating the problem as a problem for flows in electrical networks. Duijvestijn~\cite{Du78} discovered the unique perfect decomposition of a square into the smallest number of squares: it consists of $21$ pieces. It is not hard to see that in $3$ and higher dimensions no perfect tilings exist~\cite{BSST40}.

\smallskip

Fine and Niven~\cite{FN46} and Hadwiger raised the problem of characterizing, for a fixed $d\ge 2$, the set $S_d$ of all integers $n$ such that the $d$-dimensional unit cube can be decomposed into $n$ not necessarily pairwise incongruent cubes. Obviously, $i^d\in S_d$ for every positive integer $i$. Hadwiger observed that no positive integer smaller than $2^d$, or larger than $2^d$ but smaller than $2^d+2^{d-1}$, belongs to $S_d$. On the other hand, Pl\"uss~\cite{P72} and Meier~\cite{M74} showed that for any $d\ge 2$, there exists $n_0(d)$ such that all $n\ge n_0(d)$ belong to $S_d$. It is known that $n_0(2)=6$ and it is conjectured that $n_0(3)=48$ (see~\cite{CFG91}, problems C2, C5). The best known general upper bound $n_0(d) \leq (2d-2)((d+1)^d-2)-1$ is due to Erd\H os~\cite{E74}. It is conjectured that this can be improved to $n_0(d)\le c^d$ for an absolute constant $c$.

\smallskip

To show that $n_0(2)\le 6,$ consider arbitrary decompositions of the square into $6, 7,$ and $8$ smaller squares; see Fig. 1.

\bigskip
\medskip
\begin{figure}[h]
\centering
\smallskip
\includegraphics[height=1in]{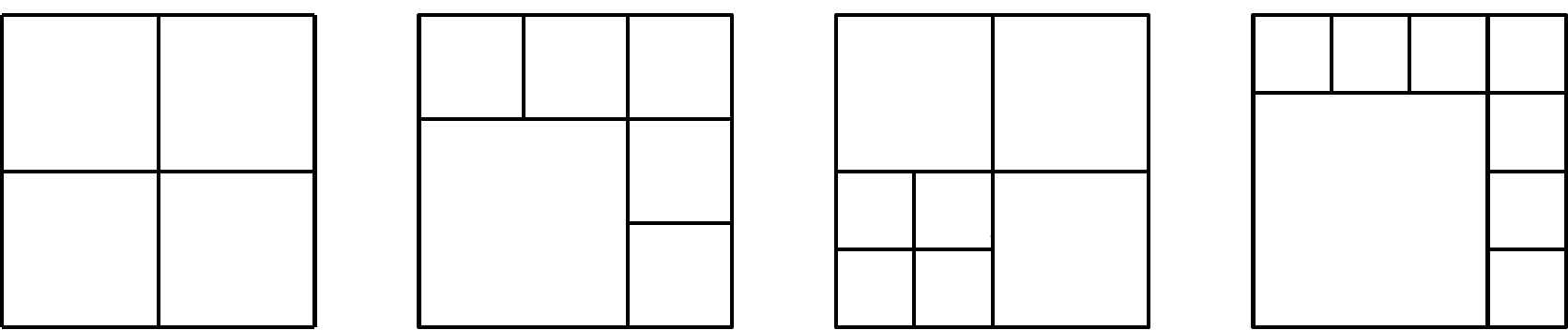}
\smallskip
\caption{Decompositions of the square into $4$, $6$, $7$, and $8$ squares. There exists no decomposition into $5$ squares.}
\medskip
\end{figure}
\medskip

Notice that from any decomposition into $n$ squares, one can easily obtain a decomposition into $n+3$ squares by replacing one of the squares, $Q$, with {\em four} others whose side lengths are half of the side length of $Q$. If we are careless, during this process we may create squares of many different sizes. In particular, for most values of $n$, the ratio of the side length of the largest square of the decomposition to the the side length of the smallest square is at least $2$.

\smallskip

Amram Meir asked many years ago whether for any $d\ge 2, \eps>0$, and for every sufficiently large $n\ge n_0(d,\eps)$, there exists a decomposition of a $d$-dimensional cube into $n$ smaller cubes such that the above ratio is smaller than $1+\eps$. The aim of this paper is to answer this question in the affirmative. In Section 2, we give a simple construction in the plane, which does not seem to generalize to higher dimensions.
\bigskip

\medskip
\noindent{\bf Theorem 1.} {\em For any integer $n\ge 6$ that is not a square number, there exists a tiling of a square with smaller squares of two different sizes such that the ratio of their side length tends to $1$ as $n\rightarrow\infty$.}
\medskip

Of course, if $n$ is a square number, then any square can be decomposed into precisely $n$ smaller squares of the same size.
\smallskip

In Section 3, we review some elementary number-theoretic facts needed for the proof of the following theorem, which will be established in Section 4.

\medskip

\noindent{\bf Theorem 2.} {\em For any integer $d\ge 2$ and any $\eps>0$, there exists $n_0=n_0(d,\eps)$ with the following property. For every $n\ge n_0$, the $d$-dimensional unit cube can be decomposed into $n$ smaller cubes such that the ratio of the side length of the largest subcube to the side length of the smallest one is at most $1+\eps$. Moreover, for every $n\ge n_0$, there is a decomposition with the required properties, using subcubes of at most $d+2$ different side lengths.}
\medskip

\begin{figure}[h]
\centering
\smallskip
\includegraphics[height=1.5in]{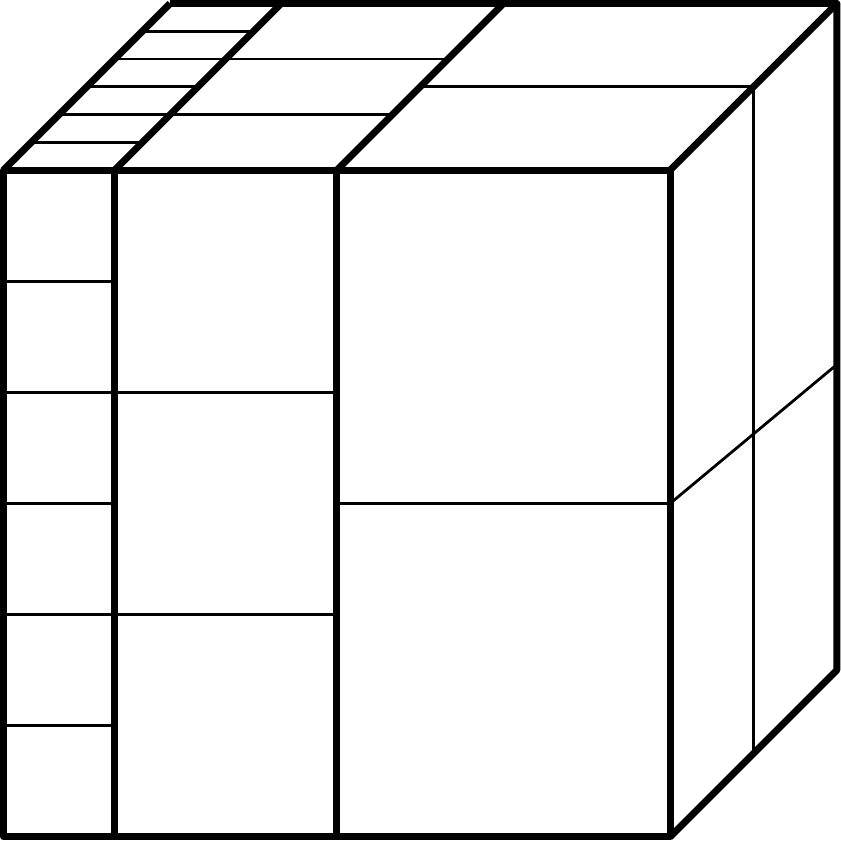}
\smallskip
\caption{A decomposition of the $3$-dimensional cube into $49$ cubes.}
\medskip
\end{figure}

The idea of the proof of Theorem 2 is the following. If $n$ is of the form $a^{3d}$ for some positive integer $a$, then we first divide the unit cube into $a^{2d}$ {\em small} cubes of side length $1/a^2$. In the second step, we subdivide each small cube into $a^d$ even smaller subcubes of side length $1/a^3$, and we are done. We call these subcubes {\em tiny}. If $n$ is not of this special form, we have $a^{3d}<n<a^{3(d+1)}$ for some $a$. In this case, after the first step, the number of tiny cubes will be smaller than $n$. Therefore, in the second step, we have to subdivide some of the small cubes into {\em slightly more} than $a^d$ equal subcubes. Unfortunately, this simple strategy does not necessarily work: the number of small cubes that need to be subdivided into more than $a^d$ pieces may exceed $a^{2d}$, the total number of small cubes produced in the first step. To overcome this difficulty, in addition, a certain number of small cubes will be subdivided into {\em fewer} than $a^d$ tiny cubes. The details are worked out in Section 4.
\smallskip

In the last section, we prove that for any $d\ge 3$ and any sufficiently large $n$ depending on $d$, it is possible to tile a cube with precisely $n$ smaller cubes of at most {\em three} different sizes; see Theorem~\ref{pro}. We close the paper with a few open problems.

\section{The planar case -- Proof of Theorem 1}

For convenience, we start with a simple observation that allows us to disregard small values of $n$ and to suppose in the rest of the argument that $n$ is sufficiently large.

\begin{lemma}\label{L0}
For every integer $n\ge 6$, a square can be tiled with $n$ smaller squares of at most {\em two} different sizes.
\end{lemma}

\begin{proof} Consider the first, second, and last tilings depicted in Fig. 1. They can be generalized as follows. For every integer $k\ge 1$, take the unit square and extend it to a square $S_k$ of side length $1+\frac1k$ by adding $2k+1$ squares of side length $\frac1k$ along its upper side and right side. We obtain a tiling with $2k+2$ squares of two different sizes, unless $k=1$ and all $4$ squares are of the same size.

By dividing the original unit square of this tiling into $4$ equal squares, we obtain a tiling of $S_k$ with $2k+5$ squares of side lengths $\frac12$ and $\frac1k$. As $k$ runs through all positive integers, we obtain tilings with $n$ squares, for $n=4$ and all $n\ge 6$.
\end{proof}
\medskip

Next, we write every positive integer $n>36$ in a special form.

\begin{claim}\label{claim1}
Let $n$ be a positive integer satisfying $a^2<n<(a+1)^2$ for some positive integer $a$. Then there exists an integer $b$ such that $a<b\le 2a$ and

(i)\;\;\; either\;\; $n=a^2+b$

(ii)\;\; or\;\; $n=(a+1)^2-b$.
\end{claim}

\noindent{\bf Proof:} If $n\le a^2+a$, then (ii) holds. If $n>a^2+a$, then (i) is true. \hfill $\Box$
\medskip

\begin{claim}\label{claim2}
Let $a$ and $b$ be positive integers with $b>a\ge 6$. Then there exists an integer $m$ with $\frac{b-2}4\le m\le\frac{b-1}2$ such that precisely one of the following conditions is satisfied.

(i)\;\;\;\; $b=2m+1=(m+1)^2-m^2,$

(ii)\;\;\;  $b=4m=(m+1)^2-(m-1)^2,$

(iii)\;\;   $b=4m+2=2(m+1)^2-2m^2.$ \hfill $\Box$
\end{claim}

\noindent{\bf Proof of Theorem 1:} In view of Lemma~\ref{L0}, it is sufficient to prove Theorem 1 for sufficiently large $n$. From now on, suppose that $n\ge 36$. We can also assume that $n$ is not a perfect square, otherwise, the statement is trivially true. Then we have $a^2<n<(a+1)^2$, for some $a\ge 6$. Let $b$ be the integer whose existence is guaranteed by Claim~\ref{claim1}, and write $b$, as in Claim~\ref{claim2}, in the form (i), (ii), or (iii). Note that $a, b,$ and $m$ are uniquely determined, and we clearly have $b>a>\sqrt{n}-1\rightarrow\infty$ and $m\ge\frac{b-2}4\rightarrow\infty$, as $n\rightarrow\infty$.
\medskip

\begin{figure}[h]
\centering
\smallskip
\includegraphics[height=2.3in]{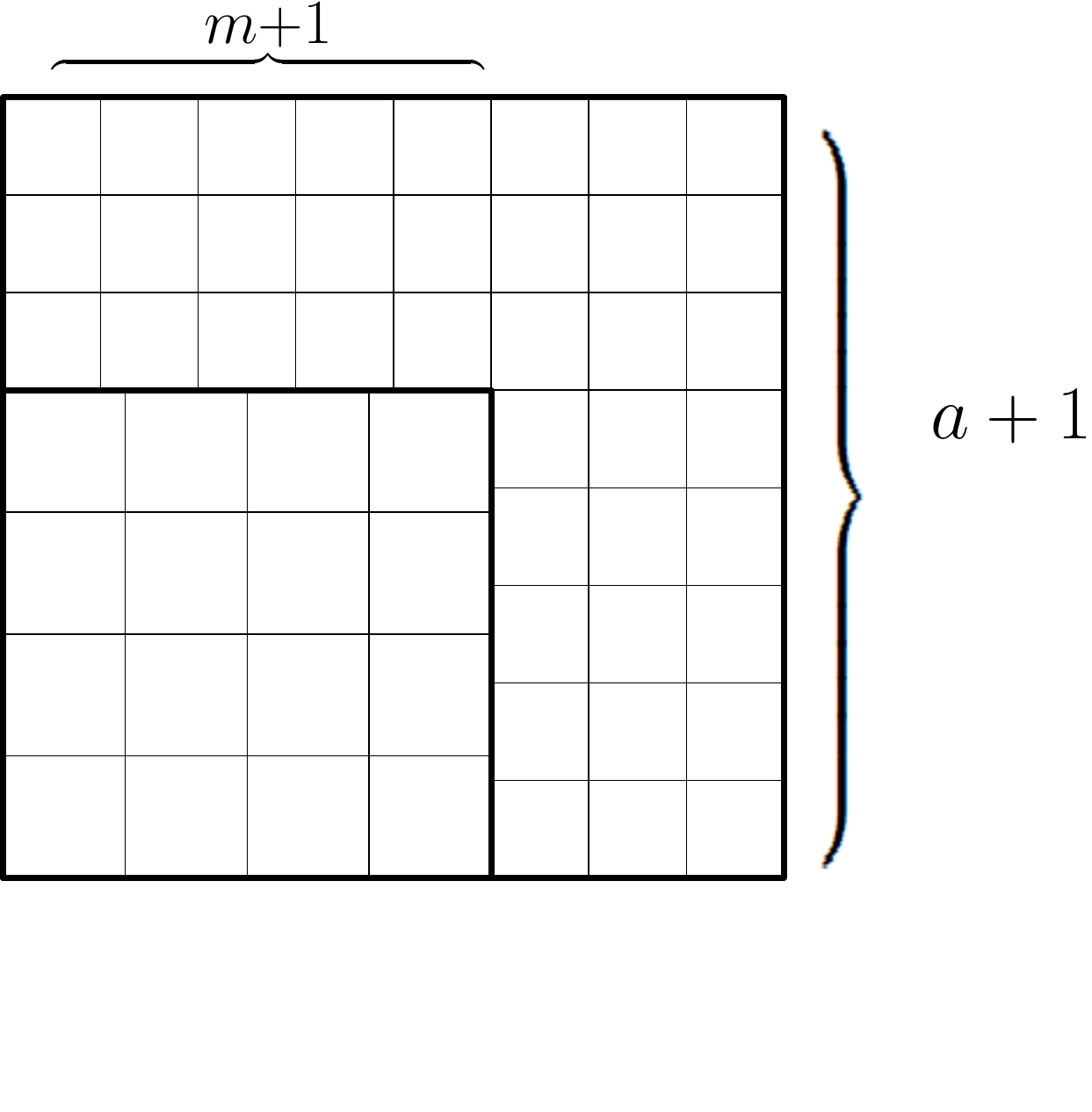}
\caption{Illustration to {\bf Case 1}. The parameters are $n=55,$ $a=7,$ $b=9,$ and $m=4,$ so that we have $n=(a+1)^2-(m+1)^2+m^2.$ }
\end{figure}

\begin{figure}[h]
\centering
\smallskip
\includegraphics[height=2.3in]{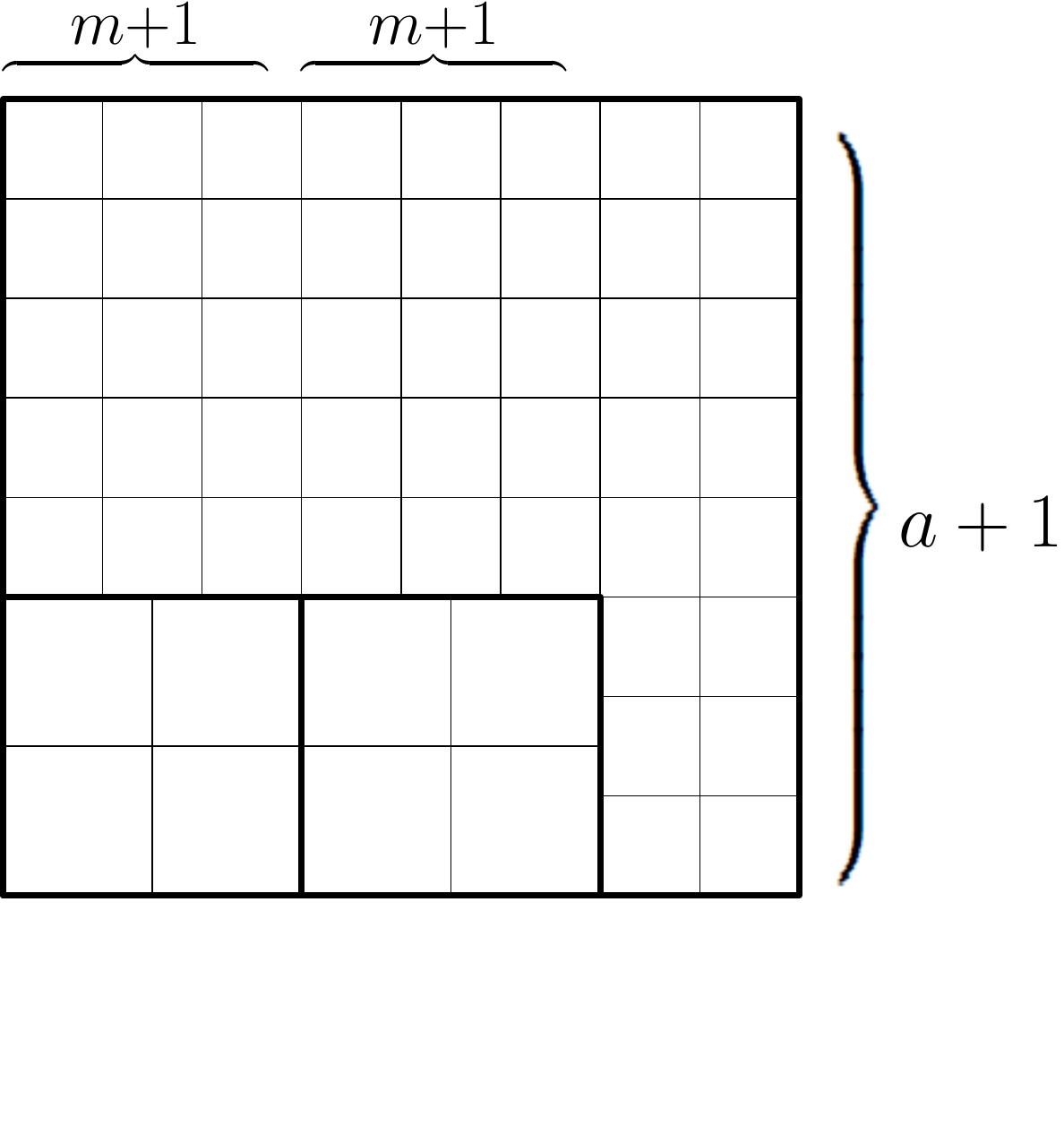}
\caption{Illustration to {\bf Case 2}. The parameters $n=54,$ $a=7,$ $b=10,$ and $m=2,$ so that we have $n=(a+1)^2-2(m+1)^2+2m^2.$ }
\end{figure}

If $b$ satisfies conditions (i) or (ii) in Claim~\ref{claim2}, we obtain $n$ as an expression consisting of three squares, two with positive signs and one with a negative sign. More precisely, we have
\begin{equation}\label{keplet1}
n=p^2-q^2+r^2\;\;\;\; {\rm with}\;\;\;\; p>q,
\end{equation}
where $p\in\{a,a+1\},$ and $q, r\in\{m-1,m,m+1\}$.

\smallskip

If $b$ satisfies condition (iii) in Claim~\ref{claim2}, we get
\begin{equation}\label{keplet2}
n=p^2-2q^2+2r^2\;\;\;\; {\rm with}\;\;\;\; p\ge 2q,
\end{equation}
where $p\in\{a,a+1\},$ and $q, r\in\{m,m+1\}$.
\medskip

We construct slightly different tilings in the above two cases.

\smallskip

\noindent{\bf Case 1:} If (\ref{keplet1}) holds, take a $p\times p$ square divided into unit squares, and replace a $q\times q$ part of it by a tiling consisting of $r^2$ squares of side length $\frac{r}{q}$. See Fig. 3.

\smallskip
\noindent{\bf Case 2:}
If  (\ref{keplet2}) holds, take a $p\times p$ square divided into unit squares, choose two disjoint $q\times q$ subsquares in it, and tile each of them with $r^2$ squares of side length $\frac{r}{q}$. See Fig. 4.

In both cases, we obtain a tiling of a large square that consists of precisely $n$ smaller squares such that the ratio of their side lengths is at most $\frac{m+1}{m-1}=1+O(\frac1{\sqrt{n}})$. This completes the proof of Theorem 1. \hfill $\Box$

\begin{problem}
Let $\rho(n)$ denote the smallest ratio of the maximum side length of a square to the minimum side length of a square over all decompositions of a square into $n$ smaller ones. Describe the asymptotic behavior of the function $$\limsup_{n\rightarrow\infty}\rho(n)-1.$$
\end{problem}

It follows from the above proof that $\rho(n)-1=O(\frac1{\sqrt{n}})$, but we have no good lower bound.

\section{Number-theoretic preliminaries}

Before turning to the proof of Theorem 2, in this short section we collect and prove some simple facts we need from elementary number theory.

\begin{lemma}\label{L1}
Let $a_1<\ldots<a_r$ and $b_1<\ldots<b_s$ be positive integers whose greatest common divisor is $1$.

(i) For every integer $k$, there exist integers $x_1,\ldots, x_r, y_1,\ldots, y_s\ge 0$ with $$\sum_{i=1}^rx_ia_i-\sum_{j=1}^sy_jb_j=k.$$

(ii) Moreover, we can assume that $\max_{i}x_i<b_s$ or $\max_{j}y_j<a_r$.
\end{lemma}

\begin{proof}
Part (i) goes back to Euclid. As for part (ii), if $x_i\ge b_s$ and $y_j\ge a_r$ for some $i$ and $j$, then we can replace $x_i$ with $x_i-b_j$ and $y_j$ with $y_j-a_i$, and (i) continues to hold. Repeating this step, if necessary, (ii) follows.
\end{proof}

In the last section, we will also use the following well known statement, due to Sylvester~\cite{Sy84}.

\begin{lemma}\label{L*}
Let $a_1$ and $a_2$ be positive integers whose greatest common divisor is $1$. Then for every integer $k\ge (a_1-1)(a_2-1)$, there exist integers $x_1, x_2\ge 0$ with $k=x_1a_1+x_2a_2$. \hfill $\Box$
\end{lemma}

\begin{lemma}\label{L2}
Let $d\ge 2$ be an integer and $p$ a prime. Then, for every fixed integer $m$, there exists $t,\; 1\le t\le d$ such that $p$ does not divide $(m+t)^d-m^d$.
\end{lemma}

\begin{proof}
Consider the polynomial $q(x)=(x+1)^d-x^d$ of degree $d-1$. We have $q(p)\equiv 1\; \mod p.$ Therefore, $q(x)$ is not the zero-polynomial and it has at most $d-1$ roots over GF$(p)$. Consequently, at least one of the numbers $q(m), q(m+1),\ldots, q(m+d-1)$ is not divisible by $p$. If $p$ does not divide $q(m)$, we are done. Otherwise, suppose that $p$ divides $q(m),\ldots,q(m+t-2)$, but does not divide $q(m+t-1)$ for some $t, 2\le t\le d$. Then we have
$$(m+t)^d-m^d=\sum_{i=0}^{t-1}q(m+i)\equiv q(m+t-1)\not\equiv 0\;\mod p,$$
as required.
\end{proof}

\begin{corollary}\label{Cor2}
For any integers $d, m>0$, the greatest common divisor of the numbers $(m+1)^d-m^d, (m+2)^d-m^d,\ldots, (m+d)^d-m^d$ is $1$. \hfill$\Box$
\end{corollary}

\section{Proof of Theorem 2}

Let $\eps$ be a (small) positive number, which will be fixed throughout this section.

\smallskip

Suppose first that $n=a^{3d}$ for a positive integer $a$. Dividing the unit cube into $a^{2d}$  {\em small} cubes of side length $1/a^2$, and then each small cube into $a^d$ {\em tiny} subcubes of the same size, we are done.

\smallskip

If $n$ is not of the above special form, we have $a^{3d}<n<(a+1)^{3d}$ for some integer $a>0$. Suppose that $n$ is so large that we have $a>d(d+1)$ and
\begin{equation}\label{formula1}
(a-1)(1+\eps)>a+d
\end{equation}

Using the assumption $a>d(d+1)\ge 6$, we have that $a^{2d}(a+4)^d>(a+1)^{3d}$. Thus, there exists an integer $c,\; 0\le c\le 3,$ with
\begin{equation}\label{formula2}
a^{2d}(a+c)^d\le n<a^{2d}(a+c+1)^d.
\end{equation}
Let $m:=a+c$.

\smallskip

We now construct a tiling of the unit cube with $n$ smaller cubes, each of side length $\frac1{a^2(m+i)}$ for some $i,\;-1\le i\le d$. Now these smaller cubes will be called {\em tiny}.

In the first step, divide the unit cube into $a^{2d}$ {\em small} cubes of side length $1/a^2$. If we subdivided each small cube into $m^d$ subcubes of side length $1/(a^2m)$, we would obtain only $a^{2d}m^d\le n$ {tiny} cubes of the same size. In order to obtain precisely $n$ tiny cubes, we need to increase their number by $$k:=n-a^{2d}m^d.$$ To achieve this, for $i=1, 2,\ldots, d,$ we pick $x_i$ small cubes, and instead of subdividing them into $m^d$ tiny subcubes, we subdivide them into $(m+i)^d$ ones. This results in an increase of $\sum_{i=1}^dx_ia_i$, where $$a_i= (m+i)^d-m^d.$$
However, if $k$ is relatively small (for example, if $k=1$), we may not be able to write $k$ in the form $\sum_{i=1}^dx_ia_i$. Therefore, we also select $y_1$ small cubes, different from the previously picked ones, and instead of subdividing them into $m^d$ pieces, we subdivide them into only $(m-1)^d$ tiny subcubes of the same size. This will reduce the number of tiny subcubes by $y_1b_1$, where
$$b_1=m^d-(m-1)^d.$$

By Corollary~\ref{Cor2}, the greatest common divisor of $a_1,\ldots,a_d,b_1$ is $1$. Applying Lemma~\ref{L1} with $r=d, s=1$, we can find nonnegative numbers $x_1,\ldots,x_d,y_1$ such that
\begin{equation}\label{formula3}
\sum_{i=1}^dx_ia_i-y_1b_1=k,
\end{equation}
as required. If, in addition, we manage to guarantee that
\begin{equation}\label{formula4}
x_1+\ldots+x_d+y_1\le a^{2d},
\end{equation}
then we are done, because there are sufficiently many small cubes in which the required replacements can be performed.

\smallskip

In what follows, we show how to find a solution of~(\ref{formula3}), for which condition~(\ref{formula4}) is satisfied. Start with any solution of~(\ref{formula3}) and, as long as possible, repeat the following two steps, producing other solutions.
\begin{enumerate}
\item If $x_i\ge b_1$ for some $i,\; 1\le i\le d,$ and $y_1\ge a_d$, then replace $x_i$ by $x_i-b_1$ and $y_1$ by $y_1-a_d$.
\item If $x_i\ge a_d$ for some $i,\; 1\le i < d$, then replace $x_i$ by $x_i-a_d$ and $x_d$ by $x_d+a_i$.
\end{enumerate}
Both of these operations, independently, can be performed at most a bounded number of times. Thus, the above procedure will terminate after a finite number of steps. We claim that the resulting solution satisfies~(\ref{formula4}).

\smallskip

It follows from $m\ge a>d(d+1)$ that
\begin{equation}\label{formula5}
a_d=(m+d)^d-m^d=\left((1+\frac{d}{m})^d-1\right)m^d<(e-1)m^d.
\end{equation}
\noindent We also have
\begin{equation}\label{formula5+}
b_1=m^d-(m-1)^d\le (a+3)^d-(a+2)^d<(e-1)a^d
\end{equation}
and
\begin{equation}\label{formula6}
m^d\le(a+3)^d=\left(1+\frac{3}{a}\right)^da^d<\left(1+\frac3{d(d+1)}\right)^da^d<ea^d.
\end{equation}

We distinguish two cases.
\medskip

\noindent{\textsc{Case 1:}} $y_1\ge a_d$

Since we cannot perform step 1, we have $x_i<b_1$ for all $i\; (1\le i\le d)$. From~(\ref{formula3}), we obtain
$$y_1b_1 <\sum_{i=1}^dx_ia_i<\sum_{i=1}^db_1a_i\le db_1a_d.$$
In view of~(\ref{formula5}) and~(\ref{formula6}), this yields that
$$y_1<da_d<d(e-1)m^d<de(e-1)a^d.$$
Taking~(\ref{formula5+}) into account, in this case we have
$$x_1+\ldots+x_d<db_1<d(e-1)a^d.$$
Combining the last two inequalities, we get
$$x_1+\ldots+x_d+y_1<d(e-1)a^d+de(e-1)a^d=d(e^2-1)a^d<a^{2d},$$
so~(\ref{formula4}) holds, as required.

\medskip

\noindent{\textsc{Case 2:}}  $y_1 < a_d$

Since we cannot perform step 2 of the algorithm, we have $x_i<a_d$ for all $i,\; 1\le i<d$. Now, using~(\ref{formula5}) and~(\ref{formula6}), we obtain
\begin{equation}\label{formula7}
x_1+\ldots+x_{d-1}+y_1<da_d<d(e-1)m^d<de(e-1)a^d.
\end{equation}

It remains to bound $x_d$ from above. From~(\ref{formula3}), we have $x_da_d\le k+y_1b_1$. Since $k=n-a^{2d}m^d,$ in view of~(\ref{formula2}), we get
$$x_da_d<a^{2d}((m+1)^d-m^d)+y_1b_1<a^{2d}((m+1)^d-m^d)+a_db_1.$$
Hence,
\begin{equation}\label{formula8}
x_d-b_1<a^{2d}\frac{(m+1)^d-m^d}{a_d}=a^{2d}\frac{(m+1)^d-m^d}{(m+d)^d-m^d}<a^{2d}\frac{1}{d}.
\end{equation}
The last inequality follows from the fact that
$$(m+d)^d-m^d=\sum_{i=0}^{d-1}((n+i+1)^d-(n+i)^d)>d((n+1)^d-n^d).$$
Adding~(\ref{formula7}),~(\ref{formula8}), and~(\ref{formula5+}), and using that $a>d(d+1)$, we conclude that
$$x_1+\ldots+x_d+y_1<de(e-1)a^d+a^{2d}\frac{1}{d}+(e-1)a^d$$
$$=a^{2d}\left(\frac1d+\frac{(de+1)(e-1)}{a^d}\right)<a^{2d}.$$
Thus, in both cases,~(\ref{formula4}) is true. This completes the proof of the Theorem 2. \hfill$\Box$

\section{Concluding remarks, open problems}

The proof in the previous section shows that, for any sufficiently large $n$, the $d$-dimensional unit cube can be tiled with $n$ smaller cubes having at most $d+2$ different sizes. Moreover, we can even require these smaller cubes to be of roughly the same size. Lemma~\ref{L0} states that for $d=2$, {\em two} different sizes suffice. For any $d$ larger than $2$, if we do drop the condition that the ratio of the sizes of the largest and smallest cubes must tend to $1$, one can decompose a cube into smaller cubes of only {\em three} different sizes.

\begin{theorem}\label{pro}
For any $d\ge 3$, there is an $n_0=n_0(d)$ with the following property. For any integer $n\ge n_0(d)$, there exists a tiling of a cube with precisely $n$ smaller cubes, each of side length $1$, $\frac12$, or $\frac1{2^d-1}$.
\end{theorem}

\begin{proof}
Since $(2^d-1)^d$ is a multiple of $2^d-1$, the numbers $(2^d-1)^d-1$ and $2^d-1$ are relatively prime. Thus, we can apply Lemma~\ref{L*} with $a_1=2^d-1$ and $a_2=(2^d-1)^d-1$. It implies that every integer $k\ge 2^{(d+1)d}$ can be expressed as
\begin{equation}\label{felbontas}
k=x_1(2^d-1)+x_2((2^d-1)^d-1),
\end{equation}
for suitable integers $x_1, x_2\ge 0$.

Suppose that $n>2^{(d+3)d}$. Then we can choose an integer $a\ge 2^{d+3}$ such that $a^d\le n<(a+1)^d$. Consider an $(a-1)\times (a-1)$ cube $C$ of side length $a-1$, and decompose it into $(a-1)^d$ unit cubes. Set $k:=n-(a-1)^d$. Observe that
$$k\ge a^d-(a-1)^d\ge (2^{d+3})^d-(2^{d+3}-1)^d\ge d(2^{d+3}-1)^{d-1}\ge 2^{(d+1)d}.$$
Therefore, we can write $k$ in the form (\ref{felbontas}), where $x_1$ and $x_2$ are nonnegative integers and, clearly,
\begin{equation}\label{felso}
x_1+x_2\le\frac{k}{2^d-1}=\frac{n-(a-1)^d}{2^d-1}<\frac{(a+1)^d-(a-1)^d}{2^d-1}<(a-1)^d.
\end{equation}
When we subdivide a unit cube in $C$ into cubes of side length $\frac12$, the number of cubes increases by $2^d-1$. When we subdivide a unit cube in $C$ into cubes of side length $\frac1{2^d-1}$, the number of cubes in the tiling increases by $(2^d-1)^d-1$. According to (\ref{felso}), there are enough unit cubes in $C$ to make $x_1$ subdivisions of the first kind and $x_2$ subdivisions of the second kind. The resulting tiling consists of precisely $n$ cubes, as required.
\end{proof}

The last theorem can be generalized in possibly two different ways.

\begin{problem}\label{P1}
Can one find, for every sufficiently large positive integer $n$, {\em two} $3$-dimensional cubes with side lengths smaller than $1$ such that the unit cube can be tiled with their congruent copies?
\end{problem}

\begin{problem}\label{P2}
Does there exist, for every sufficiently large positive integer $n$, a tiling of the unit cube with $n$ smaller cubes having at most {\em three} (or at most {\em two}) different side lengths so that the ratios of these side length tend to $1$, as $n\rightarrow\infty$?
\end{problem}

Theorem 2 shows that in $3$-dimensional space one can find tilings consisting of copies of {\em four} different cubes such that the ratios of their side lengths tend to $1$. However, in higher dimensions the number of different cubes used in our construction grows.

\begin{problem}\label{P3}
Does there exist an absolute constant $k$ such that for any $d\ge 3$ there exists $n_0(d)$ satisfying the following property? For every $n\ge n_0(d)$, there is a tiling of a cube in $\mathbb{R}^d$ with precisely $n$ cubes of at most $k$ different side lengths.
\end{problem}

\end{document}